\documentclass[authoryear,preprint,review,10pt]{elsarticle}
\usepackage[british,UKenglish,USenglish,english,american]{babel}
\usepackage[utf8]{inputenc}
\usepackage{dsfont}
\usepackage{cleveref}

\usepackage[Conny]{fncychap}			
\usepackage{fancyhdr}
\usepackage{fancybox}
\usepackage{shadethm}
\usepackage{lastpage}

\usepackage{amsfonts}
\usepackage{amsmath}
\usepackage{amssymb}
\usepackage{latexsym}

\usepackage[hmargin=2.5cm,vmargin=2.5cm]{geometry}
\usepackage{caption}
\usepackage{float}
\usepackage{color}
\usepackage{graphics}
\usepackage{multicol}

\newif\ifpdf

\ifx\pdfoutput\undefined
   \pdffalse
\else 
   \ifnum\pdfoutput=0
      \pdffalse
   \else
      \pdfoutput=1 \pdftrue
   \fi
\fi

\usepackage{ntheorem}
\newtheorem{lemma}{Lemma}
\newtheorem{theorem}{Theorem}
\newtheorem{definition}{Definition}
\newtheorem{example}{Example}
\newtheorem*{proof}{Proof:}
\newtheorem{remark}[theorem]{Remark}
\newtheorem{proposition}{Proposition}
\def\R{{\mathbb R}}
\def\P{{\mathbb P}}

\begin{document}
\begin{frontmatter}
	\title{ A note on Harnack and Transportation inequalities For
	Stochastic Differential Equations with reflections. }
	
	

	
	\author[1]{Brahim Boufoussi}
	\ead{boufoussi@ucam.ac.ma}
	
	\author[1]{Soufiane Mouchtabih}
	\ead{soufiane.mouchtabih@gmail.com}
	
	\address[1]{ Department of Mathematics, Faculty of Sciences Semlalia, Cadi Ayyad University, 2390 Marrakesh, Morocco}

\begin{abstract}
We establish transportation cost inequalities, with
respect to the uniform and $L_2$-metric, on the path space of
continuous functions, for laws of solutions of
stochastic differential equations with reflections. We also consider the
case of stochastic differential equations involving local times.
Harnack inequalities for the associated semigroups are also established.    
\end{abstract}
	

\end{frontmatter}
\section{Introduction}
 
Let $(E,d)$ be a metric space equipped with a $\sigma-$field $\mathcal{B}$
such that $d(.,.)$ is $\mathcal{B}\times\mathcal{B}$ measurable.
Given $p\geq 1$ and two probability measures $\mu$ and
$\nu$ on $E$, we define the Wasserstein distance of order $p$
between $\mu$ and $\nu$ by
$$\displaystyle W_p^d(\mu,\nu)=\inf_{\pi \in \Pi (\mu,\nu)}
\left(\int_{E\times E} d(x,y)^p d \pi(x,y)\right)^{1/p} ,$$
where $\Pi (\mu,\nu)$ is the set of all probability measures on the product
space $E \times E$ with marginals $\mu$ and $\nu$. The relative entropy of $\nu$ with respect
to $\mu$  is defined as
$$
H(\nu / \mu)=
\left\{
\begin{array}{rl}
\int_E \ln \frac{d \nu }{d \mu}\,d\nu,&  \mbox{\,\,if\,\,} \nu \ll \mu \\
+\infty & \mbox{\,\,otherwise}\,.
\end{array}
\right.
$$

The probability measure $\mu$ satisfies the $L^p-$transportation inequality
on $(E,d)$ if there exists a constant $C\geq 0$, such that for any
probability measure $\nu$,
$$W_p^d(\mu,\nu)\leq \sqrt{2C H(\nu / \mu)}.$$
We shall write $\mu \in T_p (C)$ for this relation.
It's well known that the cases $"p=1"$ and $"p=2"$ are the most interesting cases. 
$T_1(C)$ is related to concentration of measure phenomenon and
well characterized, as it was shown by \cite{djellout} using
preliminary results obtained in \cite{bob}.
Since Talagrand's paper 
(see \cite{talag}), where $T_2(C) $ inequality has been established for 
Gaussian measure, several works have emerged. \cite{feyel1} and (2004) generalized Talagrand's inequality to
an abstract Wiener space. Moreover, many different arguments have been developed to establish the
transportation inequalities. The most used method is the Girsanov transformation
argument, introduced in \cite{talag} and efficiently applied by many authors,
see, e.g., \cite{Wuzhang} for infinite-dimensional dynamical systems,
\cite{ustunel} for multi-valued SDEs and
singular SDEs, \cite{saus} for SDEs driven
by a fractional Brownian motion and \cite{li2} for stochastic delay evolution equations driven by fractional Brownian motion. Recently, \cite{Riedel} investigated the transportation inequality for the law of SDE driven by general Gaussian processes by using Lyons' rough paths theory. The first author studied in \cite{boufoussi} $T_2(C)$ inequality, with respect to $L^2 -$metric, for the stochastic heat equation driven by space-time white noise and driven by fractional noise.\\  
Even if $ T_2(C) $ is not well characterized it has many interesting properties. $T_2(C)$ is stronger than $ T_1(C) $ and it has the dimension free tensorization property. The property $T_2(C)$ is also intimately linked to many other functional properties such as concentration of measure phenomenon, Poincar\'e inequality, logarithmic Sobolev inequality and 
Hamilton-Jacobi equations. In their famous paper
\cite{otto villani 2000} showed that in a smooth
Riemannian setting, the logarithmic Sobolev inequality implies
$T_2(C)$, whereas $T_2(C)$ implies the Poincar\'e's inequality.   

In this paper, we first consider the following
reflected stochastic differential equations (RSDEs):
\begin{equation}\label{reflectedeq}
dX(t)=b(X(t))dt+\sigma(X(t))dB(t)-d\eta(t)
\,;\,\,\,\,\, X(0)=x\in{\mathcal O}\,,
\end{equation}
where $\eta(t)$ is a process with bounded variations forcing $X$
to stay inside a given regular domain $\mathcal O\subset\R^d$.
There is a rich literature on RSDEs. They  arise as a model of
some phenomena with constraints, and are useful in a variety domains of
applications, such as control theory, games theory and financial mathematics,
see for example \cite{appSoner}, \cite{appKruk}, \cite{appRamas}...). Otherwise,
they allow to give probabilistic representations for elliptic and parabolic partial
differential equations with Neumann type and/or mixed boundary conditions,
(see \cite{Freidlin}, \cite{Talay} and \cite{Brillinger2002}, etc.).
The existence and uniqueness of solutions for RSDEs of type (\ref{reflectedeq}) were first
investigated by Skorohod (see e.g. (\cite{Skorhod 1962}).  After,
many works related to reflected solutions to SDEs have been done.
Among others we cite the works of \cite{Tanaka 1979},  \cite{RSDEMen},
\cite{StrkVard}, ...etc.\\
Our first aim is to investigate the properties $T_1(C)$ and $T_2(C)$ w.r.t. uniform
and $L_2 -$metrics for the solutions of such equations.
As a consequence, we deduce a useful 
concentration inequality satisfied by the law of the solution $X$. And under
suitable regularity assumptions on the coefficients, we give an
estimation for the Wasserstein distance between the transition density
of the solution $X$ and its associated stationary distribution. Furthermore, by
means of the coupling and Girsanov transformation arguments,
we show a log-Harnack and Harnack inequalities for the operator semigroup
$$P_tf(x)=\mathbb{E}f(X^x(t))\,,\,\,\, t\geq0\,,$$
where $X^x$ is the solution of (\ref{reflectedeq}) with $X^{x}({0})=x$,
and $f$ is a bounded positive measurable function.

On the other hand, the second part of this paper concerns
stochastic differential equations involving local times (SDELs): 
\begin{equation}\label{LT0}
dX(t)=b(X(t))dt+\sigma(X(t))d{B}(t)+\int_{\mathbb{R}}\nu(dx)\,
dL_t^x(X)\,,
\end{equation}
where $\nu$ is a bounded measure on $\R$ and $L_t^x(X) $ is
the symmetric local time at $x\in\R$ of the unknown process
$(X_t)_{t\geq 0}$. These equations appeared first in the work of \cite{StroYor}
and were subsequently developed by many other authors. The necessary and sufficient conditions for pathwise uniqueness property of SDELs are given in \cite{LeGall 1983}  (see also \cite{EngeSchm}).
We know that in some special situations (when $\sigma$ and $b$ are smooth and $\nu=\beta\delta_{a} $), the solutions of
these equations are related, by mean of Feynman-Kac formula,
to parabolic differential equations with transmission conditions.
An interesting case is obtained when $b=0$,
$\sigma=1$ and $\nu=\beta\,\delta_0 $ with $\beta\in (-1,1)$,
the solution $X$ becomes the so-called Skew Brownian motion
introduced and studied in \cite{HarShep}.\\ 
Our second goal in this work is to prove transportation cost
inequalities for the solutions of SDELs, 
Moreover, as for the RSDE case we give, under a dissipativity condition,
an estimation of the Wasserstein distance between the transition
density of $X$ and its unique invariant measure. We prove also a Harnack
inequality for the corresponding semigroup. We would like to point out here that the skew Brownian is not covered by our result and that the investigation of the inequality $T_2$  remains open and interesting in mathematical
point of view.\\

The rest of this paper is organized as follows, In section 2, we recall
a result of existence and uniqueness for RSDEs \ref{reflectedeq} via
penalization method, we investigate inequalities
$T_1(C)$ and $T_2(C)$ w.r.t. uniform and $L_2$-metric for laws of
the solutions and we present a Harnack inequality for the associated
semigroup. In the last section, by using a stability argument of transportation
inequalities, we prove the property $T_2(C)$ for the law of the
solution of the equation ($(\ref{LT0})$) with respect to the $L_2$ and
uniform distance on $\mathcal{C}([0,T],\mathbb{R})$.
The Harnack inequality is also proved.
    
\section{Reflected stochastic differential equations}

Let $(B_t, t\ge 0)$ be a standard $d$-Brownian
motion ($d\geq1$), defined on a filtred probability space
$\Big(\Omega,\mathcal{F},\,(\mathcal{F})_{t\ge 0},\, P\Big)$
satisfying the usual conditions. Let $\mathcal{O}$ be a bounded convex domain
in $\mathbb{R}^d$ and $\bar{\mathcal{O}}$ denotes its
closure. Consider the normal reflected diffusion on
$\bar{\mathcal{O}}$ described as:
\begin{equation}\label{RSDE}
\left\{
\begin{aligned}
dX(t)&=b(X(t))dt+\sigma(X(t))dB(t)-d\eta(t),\\
X(0)&=x\\
\end{aligned}
\right.
\end{equation}
where $x\in \bar{\mathcal{O}}$, $ b:\mathbb{R}^d\mapsto
\mathbb{R}^d$ and $\sigma : \mathbb{R}^d\to \mathcal{M}_{d \times d}(\mathbb{R}) $
are Borel measurable functions. For a fixed horizon time $T>0$, we have:
\begin{definition}\label{def of sol of RSDE}
	A strong solution of the equation
	$(\ref{RSDE})$ on $[0,T]$ is a pair of adapted continuous
	processes $\big(X,\eta\big)$ such that:
	\begin{enumerate}
		\item $X$ takes values in the closure
		$\bar{\mathcal{O}}$ and $\eta$ has
		locally bounded variation with $\eta(0)=0$.
		\item For every  adapted continuous
		process $Y(t)$ taking values in the closure $\bar{\mathcal{O}}$
		we have 
		\begin{equation}\label{Inequality ref}
		\int_0^t <X(s)-Y(s),d\eta(s)> \ge 0
		\end{equation} 
	\end{enumerate}
	and
	\begin{equation*}
	X(t)=x+\int_0^tb(X(s))ds+\int_0^t\sigma(X(s))dB(s)-\eta(t)
	\end{equation*}
\end{definition}
We will make use of the following assumptions:

\begin{description}
	\item[$H(1)$ ]  $b$ and $\sigma$ are locally Lipschitz on
	$\mathbb{R}^d$ .
	\item[$H(2)$ ] There exists a constant $M>0$ such that for every
	$x\in \mathbb{R}^d$
	$$|b(x)|^2+|\sigma(x)|^2\le M(1+|x|^2)\,,$$
\end{description}
where $|.|$ denotes the Euclidean norm on $\R^d$.
The stochastic variational inequality $(\ref{RSDE})$ can be
approximated by means of a classic penalization method applied to
a stochastic differential equation, defined on the whole space
$\mathbb{R}^d$. Without loss of generality, one can assume that
the coefficients $b $ and $\sigma$ are defined on the whole space
$\mathbb{R}^d$, even if they need to be defined only on the closure $\bar{\mathcal{O}}$.\\
Define the penalty function $\beta(x):=x-\mathcal P_\mathcal{O}(x)$, where $P_\mathcal{O}$ is the orthogonal projection on $\bar{\mathcal{O}}$, and for every $\varepsilon > 0$, consider the stochastic differential equation:
\begin{equation}\label{penelazed equation}
\left\{
\begin{aligned}
dX_\varepsilon(t)&=b(X_\varepsilon(t))dt+\sigma(X_\varepsilon(t))dB(t)-\frac{1}{\varepsilon}\beta(X_\varepsilon(t))dt,\\
X(0)&=x.\\
\end{aligned}
\right.
\end{equation}
Since $\beta $ is Lipschitz continuous, under assumptions $H(1)$
and $H(2)$ there exists a unique strong solution $ X_\varepsilon$ of
$(\ref{penelazed equation})$. According to \cite{MenRob} we have the following convergence result 
\begin{theorem}\label{conv}
	Under the assumptions $H(1)$ and $H(2)$,
	there exists a unique solution $((X(t),\eta(t)): t\in [0,T])$
	of the stochastic variational inequality as described by Definition $(\ref{def of sol of RSDE})$. Moreover for every $T>0$, the following convergence holds in probability
	\begin{eqnarray}
	\underset{0\le t\le T}{sup}\{|X_\varepsilon(t)-X(t)|+|\eta_\varepsilon(t)-\eta(t)|\}\to 0\quad \text{as}\quad \varepsilon \to 0,
	\end{eqnarray}
	where $X_\varepsilon$ is the solution of the stochastic
	differential equation $(\ref{penelazed equation})$, and 
	\begin{equation*}
	\eta_\varepsilon(t):=\frac{1}{\varepsilon}\int_0^t\beta(X_\varepsilon(s))\, ds.
	\end{equation*}
\end{theorem}
\begin{remark}
When the convex domain $\mathcal{O}$ is unbounded, Theorem \ref{conv} is
still valid under an 
additional technical assumption, namely, there exists a point $ a\in \mathbb{R}^d$ 
and a constant $c>0$ such that
\begin{equation}\label{cond1}
\quad <x-a,\beta(x)>\ \ge\  c\,|\beta(x)|,\ \forall \, x\in \mathbb{R}^d\,.
\end{equation}
Note that if $\mathcal{O}$ is bounded then the inequality
$(\ref{cond1})$ is satisfied, we will use it in the proof of
Theorem \ref{T1(c)}.\\
\end{remark}
We recall a stability property of $ T_p(C) $ under the weak convergence
of measures, which will be useful to prove the property $T_1(C)$ for
Equation $(\ref{RSDE}) $ (see \cite{djellout}).
\begin{lemma}\label{lemma stability under week cv}
	Let $(E,d)$ be a metric separable and complete space, and $(\mu_n,\mu)_{n\in
	\mathbb{N}}$ a family of probability measures on $E$. Assume that
	$\mu_n\,\in \, T_p(C)$ for all $n\, \in\, \mathbb{N}$ and $\mu_n\to \mu$
	weakly. Then $\mu\, \in \, T_p(C)$.
\end{lemma}
We make the following assumptions: there exist $A\,,\, B>0 $ such that
\begin{equation}
\underset{x\in\,\mathbb{R}^d}{\sup}\|\sigma(x)\|_{HS}\le A,\quad
<y-x\, ,\,b(y)-b(x)>\,\le B\,(1+|y-x|^2),\ \forall x,y\, \in \mathbb{R}^d\,,
\label{Co1}
\end{equation}
where $\|.\|_{HS}$ is the Hilbert-Schmidt norm.
\begin{theorem}\label{T1(c)}
Suppose that $H(1)$, $H(2)$ and the conditions (\ref{Co1}) hold. Let
$\mathbb{P}_{X}$ be the law of the solution $ X$ of $(\ref{RSDE})$,
with initial point $X_{0}=x\in \bar{\mathcal{O}} $. Then for each  $T>0$
there exists some constant $C=C(T,A,B)$ independent of $x$ such that
$\mathbb{P}_X$
satisfies $T_1(C)$, on the space $\mathcal{C}\big([0,T],\mathbb{R}^d\big)$
equipped with the uniform metric 
$$d_\infty(\gamma_1,\gamma_2)=\underset{t\in [0,T]}{\sup}|\gamma_1(t)-\gamma_2(t)|.$$
\end{theorem}

\begin{proof}
Since by the inequality (\ref{cond1}) we have $ <x-y, \beta_{\varepsilon}(x)-
\beta_{\varepsilon}(y)> \geq 0 $ $\forall x, y $, $\forall\varepsilon$,
it is clear that
the coefficients of  the panellized equation $(\ref{penelazed equation})$
satisfy the assumptions of Corollary 4.1 
in \cite{djellout}, which ensures that 
for any $\varepsilon>0$, the law $\mathbb{P}_{X_{\varepsilon}} $ of
$ X_\varepsilon $ satisfies $T_1(C)$ for some constant
$C=C(T,A,B)$ independent of $\varepsilon$. By the stability argument of $T_1(C)$
under the weak convergence of measures, we conclude that $\mathbb{P}_{X} \in T_1(C)$.
\end{proof}
\begin{remark}
We can deduce several consequences of Theorem \ref{T1(c)}
\begin{enumerate}
\item For any Lipschitzian function $F:\mathcal{C}([0,T],\mathbb{R}^d)\to
\mathbb{R}$ we have, see Theorem 1.1 in \cite{djellout},
\begin{equation}
\mathbb{P}_X\big(F-E_{\mathbb{P}_X}F>r\big)\ \le\ \exp\big(-\frac{r^2}{2C\|F\|^2_{Lip}}\big)
\end{equation}
where, 
$$\|F\|_{Lip}=\underset{\gamma_1\ne \gamma_2}{\sup}\frac{|F(\gamma_1)-F(\gamma_2)|}
{d_\infty(\gamma_1,\gamma_2)}\,. $$
\item  We have the following concentration inequality $($see \cite{GozLeosurvey}$)$.
For all measurable $A\subset \mathcal{C}([0,T],\mathbb{R}^d)$ with $\mathbb{P}_X(A)\ge 1/2$, 
\begin{equation*}
\mathbb{P}_X(A^r)\ge 1-\exp(-\frac{r-r_0}{2C}),\quad r\ge r_0=\sqrt{2Clog(2)}
\end{equation*}
where $A^r$ is defined by
$$A^r:=\{\gamma\in \mathcal{C}([0,T],\mathbb{R}^d); d_\infty(\gamma,A)\le r\},\quad
r\ge 0.$$
\end{enumerate}
\end{remark}

In the sequel, let $\mathcal{F}_t:=\sigma(B(s),\, s\le t)\vee {\mathcal N}$,
where $\mathcal N$
is the class of all $\mathbb{P}-$ negligible sets. We will show that
the law $\mathbb{P}_X$ of the solution $X$ of the equation $(\ref{RSDE})$
satisfies $T_2(C)$ on the space
$\mathcal{C}\big([0,T],\mathbb{R}^d\big)$  with respect to the two
metrics 
$$ d_2(\gamma_1,\gamma_2)=\big(\int_0^T|
\gamma_1(t)-\gamma_2(t)|^2 dt\big)^{\frac{1}{2}}\,\,\,\text{and}\,\,\,\,\,\,
d_\infty(\gamma_1,\gamma_2)=\underset{t\in [0,T]}{\sup}|\gamma_1(t)-\gamma_2(t)|.$$
For this end, we assume the following dissipative condition:  
\begin{description}
\item[$H(3)$]\qquad  There is $\delta>0$ such that:
$$\|\sigma(x)-\sigma(y)\|^2_{HS}+2<x-y,b(x)-b(y)>\le 
-2\delta|x-y|^2,\ \forall x,y\in \mathbb{R}^d \,,$$ 
\end{description}
where $M^t$ and $tr(M)$ denote respectively the transpose and the trace
of a matrix $M$, and $|.|$ stands for the Euclidean norm on $\mathbb{R}^d$.
We will note by
$ \|\sigma\|_{\infty}=\displaystyle\sup_{x\in\R^d}\, \sup_{\|z\|\leq 1}
\vert\sigma(x)\, z\vert$.

\begin{theorem}\label{T2 w.r.t L2}
Suppose that $H(1)$, $H(2)$, $H(3)$ hold and
$\|\sigma\|_{\infty}<\infty$. Then for any initial point
$X_{0}= x\in \bar{\mathcal O}$
and $T>0$, the law $\mathbb{P}_X $ of the solution $X$ satisfies
$ T_2(\frac{\|
\sigma\|_\infty^2}{\delta})$ on $\mathcal{C}\big([0,T],\mathbb{R}^d\big)$ 
with respect to the metric $d_2$.
\end{theorem}
\begin{proof}
The first part of the proof follows the argument of \cite{djellout}. The idea is to express
the finiteness of the entropy by means of the energy of the drift arising from the Girsanov
transform of a well-chosen probability. Let $\mathbb{Q}$ be a probability
measure on
$\mathcal{C}([0,T],\mathbb{R}^d)$ such that $\mathbb{Q}\ll \mathbb{P}_X$,
we assume that
$H(\mathbb{Q}/\mathbb{P}_X)<\infty$ and we consider
$$\tilde{\mathbb{Q}}:=\dfrac{d\mathbb{Q}}{d\mathbb{P}_X} (X)\mathbb{P}$$
Clearly $\tilde{\mathbb{Q}}$ is a probability measure on $(\Omega,\mathcal{F})$ and 
\begin{eqnarray*}
\textbf{H}(\tilde{\mathbb{Q}}/\mathbb{P})&=&\int_{\Omega}\ln\big(\frac{d\tilde{\mathbb{Q}}}{d\mathbb{P}}\big)\ d\tilde{\mathbb{Q}}\\
&=&\int_{\Omega}\ln\big(\frac{d\mathbb{Q}}{d\mathbb{P}_X}(X)\big)\ \frac{d\mathbb{Q}}{d\mathbb{P}_X}(X)\ d\mathbb{P}\\
&=&\int_{\mathcal{C}\big([0,T],\mathbb{R}^d\big)}
\ln\Bigg(\frac{d\mathbb{Q}}{d\mathbb{P}_X}\Bigg)\ 
\frac{d\mathbb{Q}}{d\mathbb{P}_X}\ d\mathbb{P}_X\\
&=&\textbf{H}(\mathbb{Q}/\mathbb{P}_X)
\end{eqnarray*} 
The principal key of the proof is the following result
(see  \cite{djellout}):
{\sl
There exists a predictable process $\rho=(\rho^1(t),...,
\rho^d(t))_{0\le t\le T} $, such that 
$$\textbf{H}(\tilde{\mathbb{Q}}/\mathbb{P})=\textbf{H}(\mathbb{Q}/\mathbb{P}_X)=\frac{1}{2}\mathbb{E}_{\tilde{\mathbb{Q}}}\int_0^T|\rho(t)|^2 dt.$$
}
By Girsanov's theorem the process defined by:
$$\tilde{B}(t):=B(t)-\int_0^t \rho(s)ds$$
is a Brownian motion under $\tilde{\mathbb{Q}}$, consequently $X$ verifies
\begin{equation}
\left\{
\begin{aligned}
dX(t)&=b(X(t))dt+\sigma(X(t))d\tilde{B}(t)+\sigma(X(t))\rho(t)dt-d\eta_{X}(t),\\
X(0)&=x.\\
\end{aligned}
\right.
\end{equation}
We consider the solution $Y$ of the following equation
\begin{equation}
\left\{
\begin{aligned}
dY(t)&=b(Y(t))dt+\sigma(Y(t))d\tilde{B}(t)-d\eta_Y(t),\\
Y(0)&=x.\\
\end{aligned}
\right.
\end{equation}
Under $\tilde{\mathbb{Q}}$ and by the uniqueness argument,
the law of the process $(Y(t))_{t\in [0,T]}$ is exactly $\mathbb{P}_X$. Then, under $\tilde{\mathbb{Q}}$,
$(X,Y)$ is a coupling of $(\mathbb{Q},\mathbb{P}_X)$, then it follows that
$$\Bigg[W^{d_2}_2\Big(\mathbb{Q},\mathbb{P}_X\Big)\Bigg]^2\le \mathbb{E}_{\tilde{\mathbb{Q}}}\Big(d_2(X,Y)^2\Big)=\mathbb{E}_{\tilde{\mathbb{Q}}}\Bigg(\int_0^T|X(t)-Y(t)|^2\ dt\Bigg).$$
Now, we estimate the distance on $\mathcal{C}\big([0,T],\mathbb{R}^d\big)$ between $X$ and $Y$ with respect to the distance $d_2$.\\
With the notations:
\begin{multicols}{2}\noindent
	\begin{equation*}
	\hat{X}(t):=X(t)-Y(t)
	\end{equation*}
	\begin{equation*}
	\hat{b}(t):=b(X(t))-b(Y(t))
	\end{equation*}
	\begin{equation*}
	\hat{\sigma}(t):=\sigma(X(t))-\sigma(Y(t))
	\end{equation*}
	\begin{equation*}
	\hat{\eta}(t):=\eta_X(t)-\eta_Y(t) ,
	\end{equation*}	
\end{multicols}
	the process $\hat{X}$ satisfies the following It\^o equation
\begin{equation}
d\hat{X}(t)=\hat{b}(t)dt+\hat{\sigma}(t)d\tilde{B}(t)+
\sigma(X(t))\rho(t)dt-d\hat{\eta}(t)\,.
\end{equation}
By It\^o formula, we have 
\begin{equation}\label{equation by Ito}
d|\hat{X}(t)|^2=\Big[2<\hat{X}(t),\hat{b}(t)+\sigma(X(t))\rho(t)>+
tr(\hat{\sigma}(t)\hat{\sigma}(t)^t)\Big]dt+2<\hat{X}(t),\hat{\sigma}(t)d\tilde{B}(t)>-
2<\hat{X}(t),d\hat{\eta}(t)>\,.
\end{equation}
Using assumption $H(3)$ and the condition $(\ref{Inequality ref})$, we get 
\begin{equation*}
|\hat{X}(t)|^2\ \le\ -2\delta\int_0^t|\hat{X}(s)|^2\ ds+ 2 \int_0^t<\hat{X}(s),\sigma(X(s))\rho(s)>ds+2\int_{0}^{t}<\hat{X}(s),\hat{\sigma}(s)
d\tilde{B}(s)>\,.
\end{equation*}
By using a localization argument and the Cauchy-Schwartz inequality, we obtain for
each $\lambda >0$
\begin{equation*}
\mathbb{E}_{\tilde{\mathbb{Q}}}|\hat{X}(t)|^2\le (\lambda-2\delta)
\int_0^t\mathbb{E}_{\tilde{\mathbb{Q}}}|\hat{X}(s)|^2\ ds+
\frac{\|\sigma\|_\infty^2}{\lambda}\mathbb{E}_{\tilde{\mathbb{Q}}}
\int_0^t|\rho(s)|^2\ ds\,.
\end{equation*}
Gronwall's lemma entails 
\begin{equation}\label{inequality for TC for Kernel}
\mathbb{E}_{\tilde{\mathbb{Q}}}|\hat{X}(t)|^2\le \frac{\|\sigma\|_\infty^2}{\lambda}\mathbb{E}_{\tilde{\mathbb{Q}}}\int_0^t e^{(\lambda-2\delta)(t-s)}|\rho(s)|^2\ ds.
\end{equation}
Thus,
\begin{eqnarray*}
\Bigg[W^{d_2}_2\Big(\mathbb{Q},\mathbb{P}_X\Big)\Bigg]^2 &\le& \mathbb{E}_{\tilde{\mathbb{Q}}}\int_0^T|\hat{X}(t)|^2\ dt\\
             &\le& \frac{\|\sigma\|_\infty^2}{\lambda}\mathbb{E}_{\tilde{\mathbb{Q}}}\int_0^T\int_0^te^{(\lambda-2\delta)(t-s)}|\rho(s)|^2\ ds\ dt     \\
             &\le& \frac{\|\sigma\|_\infty^2}{\lambda} \frac{1-e^{(\lambda-2\delta)T}}{2\delta-\lambda} \mathbb{E}_{\tilde{\mathbb{Q}}}\int_0^T|\rho(s)|^2\ ds                                               
\end{eqnarray*}
Choosing $\lambda=\delta$, we get $\mathbb{P}_X$ verifies
$T_2(\frac{\|\sigma\|_\infty^2}{\delta^2})$. 
\end{proof}

\begin{remark}
It is not surprising to obtain the same constant $C$ as in the case
of non reflected diffusions (see Theorem 5.6 of \cite{djellout})), It
is due to the fact that the reflection term $\eta_X $ satisfies
$$ <X-Y, \eta_X-\eta_Y>\geq 0\,,\,\,\,\, a.s.\,; $$
for two solutions $X, Y$ of the RSDE.
\end{remark}
\begin{remark} {}
	We have notable consequences of $T_2(C)$ such as:
	\begin{enumerate}
\item For any smooth cylindrical function $F$ on $\mathcal{C}([0,T],
\mathbb{R}^d)\subset G:=L^2([0,T],\mathbb{R}^d\,; dt)$, that is 
$$F\in S=\{f(<\gamma,h_1>,...,<\gamma,h_n>); n\ge 1, h\in H,f\in C^\infty_b(\mathbb{R}^n)\}$$
where $H$ is Cameron-Martin space and $<\gamma_1,\gamma_2>=\int^T_0
\gamma_1(t)\gamma_2(t)\ dt$, we have 
\begin{equation}
Var_{\mathbb{P}_X}(F)\,\le\,\frac{\|\sigma\|_\infty^2}{\delta^2}
\int_{\mathcal{C}([0,T],\mathbb{R}^d)} \|\nabla F(\gamma)\|^2_G\,
d\mathbb{P}_X,
\end{equation}
where $Var_{\mathbb{P}_X}(F)$ is the variance of $F$ under the law
$\mathbb{P}_X$, and $\nabla F(\gamma)\in G $ is the gradient of $F$ at
$\gamma$.
\item  Let $K$ be a nonempty subset in $G$ such that  $Z(\gamma)=\underset{h\in K}{\sup}<\gamma,h>\, \in \, L^1(\mathbb{P}_X)$, then we have the following Tsirlson's type concentration inequality
\begin{equation}
\int \exp(\frac{\delta^2}{\|\sigma\|_\infty^2}\, \underset{h\in K}{\sup}\big[<\gamma,h>-\frac{\|h\|^2_G}{2}\big])\ d\mathbb{P}_X \le \, \exp\big(\frac{\delta^2}{\|\sigma\|_\infty^2}\mathbb{E}_{\mathbb{P}_X}Z\big)
\end{equation}
\end{enumerate}
\end{remark}

Let $(P_{t}(x, .))_{t\geq0}$ be
the transition probability kernels of the solution $X^{x}(t)$ with
initial point
$X^{x}(0)=x\in \bar{\mathcal{O}} $.
In the following we derive an
estimation of the Wasserstein-distance between the invariant measure of $X^x $ and its associated transition
distributions. More precisely, we have
\begin{theorem}
Under the same assumptions of Theorem $\ref{T2 w.r.t L2}$,
the following holds true: $ P_t $ admits a unique invariant probability measure
$\mu$, and 
\begin{equation}\label{decr expon pour mesure invaraint}
W_2(P_t(x,.),\mu)\le e^{-\delta t}\big(\int |x-y|^2\ d\mu(y)\big)^\frac{1}{2}\,,
\quad \forall x\in \bar{\mathcal{O}},\ t>0.
\end{equation}
\end{theorem}
\begin{proof}
	Let $X^{x}(t)$ and $X^{y}(t)$ be the solutions of equation 
	$(\ref{RSDE})$ with
	initial point $x,\ y\ \in \bar{\mathcal{O}}$ respectively. We use
	It\^o formula, assumption $H(3)$,
	condition $(\ref{Inequality ref})$ and Gronwall's lemma to obtain   
	\begin{equation*}
\mathbb{E}|X^{x}(t)-X^{y}(t)|^2\ \le\ |x-y|^2e^{-2\delta t},\quad
\forall t\ge 0\,,
	\end{equation*}
	which gives rise,
	by a classic coupling argument $($see for example \cite{Liming Wu}$)$,
	to the existence of a unique invariant probability measure of $(P_t)$
	on $\mathbb{R}^d$ satisfying
	$(\ref{decr expon pour mesure invaraint})$.\\	
\end{proof}
\begin{remark}
\begin{enumerate}
\item[]{}
 \item[1--] Similar arguments as those used in \cite{djellout} for the
 case of
 diffusions without reflections are valid to derive that
 $$P_T(x,.)\in T_2(\frac{\|\sigma\|^2_\infty}{2\delta})\,.$$
Since $ P_T(x,.)\to \mu$ as $T\to \infty$, we get by Lemma
$\ref{lemma stability under week cv}$,
that the invariant measure $\mu $ satisfies also
$ T_2(\frac{\|\sigma\|^2_\infty}{2\delta})$.
\item[2--] The following Poincar\'e inequality holds $($see Theorem 5.6,
\cite{BobGenLed}$)$, for any $g\in C^\infty_b(\mathbb{R}^d)$
	\begin{equation*}
	Var_{P_T(x,.)}(g)\ \le\ \frac{\|\sigma\|_\infty^2}{2\delta}\int_{\mathbb{R}^d}\|\nabla g(y)\|^2\ P_T(x,dy).
	\end{equation*}
\end{enumerate}
\end{remark}

Now, we investigate the Talagrand's inequality $T_{2}(C)$ for
the law $\P_{X}$ of the solution of the equation (1) with respect to the
uniform norm $d_{\infty}$.
\begin{theorem}\label{T2 w.r.t uniform metric}
	In addition of assumptions $H(1)$, $H(2)$, $H(3)$ and
	$\|\sigma\|_{\infty}<\infty$, suppose that 
	$\sigma$ is globally Lipschitzian with Lipschitz constant
	$\|\sigma\|_{Lip}$.
	Then for any $T>0$ there exists some constant
	$C=C(T, \|\sigma\|_{Lip}, \|\sigma\|_{\infty})>0$
	such that, for any initial point $X(0)=x\in\, \bar{\mathcal{O}}$,
	the law $\P_{X}$
	satisfies $T_2(C)$ on  $\mathcal{C}\big([0,T],\mathbb{R}^d\big)$,
	with respect to the uniform metric $d_\infty$. 
\end{theorem}
\begin{proof}
	Similarly to the proof of Theorem $\ref{T2 w.r.t L2}$,
	we have with the same notations 
	\begin{equation*}
	d\hat{X}(t)=\hat{b}(t)dt+\hat{\sigma}(t)d\tilde{B}(t)+
	\sigma(X(t))\rho(t)dt-d\hat{\eta}(t)\,.
	\end{equation*}
	By It\^o formula, we get 
	$$d|\hat{X}(t)|^2=\Big[2<\hat{X}(t),\hat{b}(t)+
	\sigma(X(t))\rho(t)>+tr(\hat{\sigma}(t)\hat{\sigma}(t)^t)\Big]dt+
	2<\hat{X}(t),\hat{\sigma(t)}d\tilde{B}(t)>-2<\hat{X}(t),
	d\hat{\eta}(t)>\,.$$
	By virtue of $H(3)$, the condition $(\ref{Inequality ref})$ and
	Cauchy-Schwartz inequality, we achieve for each $\lambda >0$
	\begin{equation}\label{main equation}
	\underset{s\le t}{\sup}|\hat{X}(s)|^2\le (\lambda-2\delta)\int_0^t|
	\hat{X}(s)|^2 ds\ +\ \frac{\|\sigma\|_\infty^2}{\lambda}
	\int_0^t|\rho(s)|^2ds\ +\ \underset{s\le t}{\sup}\ 2\
	\Big|\int_0^s<\hat{X}(u),\hat{\sigma}(u)d\tilde{B}(u)>\Big|\,.
	\end{equation}
	Burkholder-Davies-Gundy inequality gives, for any
	$\alpha>0$
	\begin{eqnarray*}
\mathbb{E}_{\tilde{\mathbb{Q}}}\ \underset{s\le t}{\sup}\ 2\ \Big|\int_0^s<\hat{X}(u),\hat{\sigma}(u)d\tilde{B}(u)>\Big|&\le& 6\,\mathbb{E}_{\tilde{\mathbb{Q}}}\Bigg(\int_0^t|\hat{\sigma}^t(s)\hat{X}(s)|^2\ ds\Bigg)^\frac{1}{2}\\
	&\le& 6\, \|\sigma\|_{Lip}\ \mathbb{E}_{\tilde{\mathbb{Q}}}\Bigg(\int_0^t|\hat{X}(s)|^4\ ds\Bigg)^\frac{1}{2}\\
	&\le& 6\, \|\sigma\|_{Lip}\ \mathbb{E}_{\tilde{\mathbb{Q}}}\Bigg(\underset{s\le t}{\sup}|\hat{X}(s)|^2\int_0^t|\hat{X}(s)|^2\ ds\Bigg)^\frac{1}{2}\\
	&\le& 3\, \|\sigma\|_{Lip}\  \mathbb{E}_{\tilde{\mathbb{Q}}}\Bigg(\alpha\
	\underset{s\le t}{\sup}|\hat{X}(s)|^2+\frac{1}{\alpha}
	\int_0^t|\hat{X}(s)|^2\ ds\Bigg)\,.
	\end{eqnarray*}
	Therefore, by $ (\ref{main equation}) $ together with
	the last inequality we obrain  
\begin{eqnarray*}
\mathbb{E}_{\tilde{\mathbb{Q}}}\,\underset{s\le t}{\sup}|\hat{X}(s)|^2   &\le&   3\,\|\sigma\|_{Lip}\,\alpha\, \mathbb{E}_{\tilde{\mathbb{Q}}}\,\underset{s\le t}{\sup}|\hat{X}(s)|^2+(\lambda-2\delta+\frac{3\|\sigma\|_{Lip}}{\alpha})\int_0^t \mathbb{E}_{\tilde{\mathbb{Q}}}\,\underset{u\le s}{\sup}|\hat{X}(u)|^2 du\\
	& & +\frac{\|\sigma\|_\infty^2}{\lambda}\,
	\mathbb{E}_{\tilde{\mathbb{Q}}}\int_0^t|\rho(s)|^2ds . 	
\end{eqnarray*}
Thus, choosing $ 0<\alpha < \frac{1}{3 \|\sigma\|_{Lip}}$, we get
\begin{eqnarray*}
\mathbb{E}_{\tilde{\mathbb{Q}}}\,\underset{s\le t}{\sup}|\hat{X}(s)|^2 
&\le& \dfrac{\lambda-2\delta+\frac{3\|\sigma\|_{Lip}}{\alpha}}
{1-3\alpha\|\sigma\|_{Lip}}\int_0^t \mathbb{E}_{\tilde{\mathbb{Q}}}\,
\underset{u\le s}{\sup}|\hat{X}(u)|^2\, du\,+\,\dfrac{\|\sigma\|_\infty^2}
{\lambda(1-3\alpha\|\sigma\|_{Lip})}\,\mathbb{E}_{\tilde{\mathbb{Q}}}\int_0^t|\rho(s)|^2\,ds\,.
\end{eqnarray*}	 
Let
 $$C_1=\frac{\lambda-2\delta+\frac{3\|\sigma\|_{Lip}}{\alpha}}{1-3\alpha\|\sigma\|_{Lip}}\  and\quad C_2=\dfrac{\|\sigma\|_\infty^2}
 {\lambda(1-3\alpha\|\sigma\|_{Lip})}\,.$$
 Gronwall's lemma implies 
\begin{eqnarray*}
\mathbb{E}_{\tilde{\mathbb{Q}}}\,\underset{t\le T}{\sup}|\hat{X}(t)|^2
&\le & C\, \mathbb{E}_{\tilde{\mathbb{Q}}}\int_0^T|\rho(s)|^2\,ds
\end{eqnarray*}	
Thus, 
$$\Big[W^{d_\infty}_2\big(\mathbb{Q},\mathbb{P}_X\big)\Big]^2\ \le
C\,\mathbb{E}_{\tilde{\mathbb{Q}}}\int_0^T|\rho(s)|^2\,ds\,.$$
Which proves that $\mathbb{P}_X\in T_2(C)$, where $C=C_2e^{C_1T}$.
\end{proof}
\begin{remark}
	The property $T_2(C)$ with respect to the uniform metric
	is stronger than $T_2(C)$ with respect to $L^2-$metric,
	However this gain have two costs, the first one is
	the globally Lipschitzian property of $\sigma$,
	the second one is the loss of sharpness of the constant
	as in Theorem \ref{T2 w.r.t L2}.
	Remark that if we choose
	$\lambda=2\delta $, the optimalilty is obtained with the constant
	$C=\frac{\vert\vert\sigma\vert\vert_{\infty}}{\delta}
	\, e^{36 \vert\vert\sigma\vert\vert_{Lip}^2\, T} $.
	
\end{remark}
\begin{remark}

 The following result was established in \cite{BobGenLed} on $\mathbb{R}^d$ and extended after to $\mathcal{C}([0,T],\mathbb{R}^d)$ (see \cite{villani 2003}):  
Let $F$ be a lower bounded measurable function on $\mathcal{C}([0,T],\mathbb{R}^d)$, and consider
\begin{equation*}
Q_c F(\gamma):=\underset{h\in\,\mathcal{C}([0,T],\mathbb{R}^d) }{\inf}\Big(F(\gamma+h)+\frac{1}{2c}\|h\|_{\infty}^2\Big)
\end{equation*}	
the inf-convolution on $\mathcal{C}([0,T],\mathbb{R}^d)$ with respect to metric $d_\infty$. Then $T_2(C)$ in Theorem $\ref{T2 w.r.t uniform metric}$ implies
\begin{equation*}
\mathbb{E}_{\mathbb{P}_X}\exp\Big(Q_C F\Big)\, \le \, \exp\Big(\mathbb{E}_{\mathbb{P}_X}F\Big)\,.
\end{equation*}
If in addition $F$ is Lipschitzian, since $Q_C F\ge F-\frac{C}{2}\|F\|^2_{Lip}$, we have the following concentration inequality
\begin{equation}
\mathbb{E}_{\mathbb{P}_X}\exp\big(F-\mathbb{E}_{\mathbb{P}_X}F\big)\, \le \, \exp\big(\frac{C}{2}\|F\|^2_{Lip}\big)
\end{equation}
which is similar to the one obtained by the property $T_1(C)$, but here the constant $C$ is explicit.
By Chebyshev's inequality and an optimization argument we obtain
\begin{equation*}
\mathbb{P}_X\Big(F-\mathbb{E}_{\mathbb{P}_X} F> r\Big)\,\le \, \exp\Big(-\frac{r^2}{2C\|F\|_{Lip}^2}\Big),\quad \forall\, r\, >0,
\end{equation*}
which is also valid, by an approximation argument, for unbounded Lipschitzian function $F$.
\end{remark}

In this part we establish a Harnack inequality for the
semigroup of the reflected RSDE. We shall use the same technics as those
used in \cite{wang}, which consist in constructing a coupling under a new
probability measure by Girsanov transformation.  
We need to suppose the following  additional assumptions:
$$H(4):\quad \sigma(x)^t\sigma(x)\ge \lambda I,\quad \forall x\in \mathbb{R}^d,$$
$$H(5):\quad |<\sigma(x)-\sigma(y),x-y>|\le k|x-y|,\quad \forall x,y\in
\mathbb{R}^d\,,$$
where $\lambda, k>0 $ are two real constants.
\begin{theorem}\label{Harnack1}
	\begin{enumerate}
		\item If $H(3)$ and $H(4)$ are satisfied, then log-Harnack inequality
		holds, for all $f\ge 1$, $x,y\in \mathcal{O}$
		\begin{equation*}
		P_T\log f(y)\le \log P_Tf(x) +\frac{-\delta|x-y|}{\lambda^2(1-e^{2\delta T})}.
		\end{equation*}
	\item If $H(3)$, $H(4)$ and $H(5)$ hold, then for $p>(1+\frac{k}{\lambda})^2$ and $c_p=\max\{k,\frac{\lambda}{2}(\sqrt{p}-1)\}$, then the Harnack inequality
	\begin{equation*}
	(P_T f(y))^p\le (P_T f^p(x))\exp\left[\frac{-\delta\sqrt{p}(\sqrt{p}-1)|x-y|}{2c_p((\sqrt{p}-1)\lambda-c_p)(1-e^{2\delta T})} \right] 
	\end{equation*}	
	holds for all $T>0$,$x,y\in \mathcal{O}$ and $f$ bounded positive function.
	\end{enumerate}
\end{theorem} 
\begin{proof} Exploiting in many places the fact that for two solutions $X, Y $ we have
$$ <X-Y, \eta_X-\eta_Y>\geq 0\,,\,\,\,\, a.s.\,;$$
the proof follows exactly the same lines as for the
non-reflected diffusions case.
So, we only give some ideas of the proof and refer to \cite{wang} for more
details.\\
Let $x,y\in \mathcal{O}$, $T>0$ and
$p>(1+\frac{k}{\lambda})^2$ be fixed such that $x\neq y$.
We set
	\begin{equation}\label{Theta_T}
	\theta_T:=\frac{2k}{(\sqrt{p}-1)\lambda}\in (0,2)
	\end{equation}
For $\theta\in (0,2)$, we consider
$$\xi_t=\frac{2-\theta}{-2\delta}(1-e^{-2\delta(t-T)}),\, t\in [0,T].$$
Then $\xi$ is smooth and strictly positive on $[0,T)$ such that
\begin{equation}
2+2\delta\xi_t+\xi_t'=\theta,\ t\in [0,T].
\end{equation}
 We consider the coupling
\begin{eqnarray*}
dX(t)&=& b(X(t))dt+\sigma(X(t))dB(t)-d\eta_X(t),\quad X(0)=x\\
dY(t)&=& b(Y(t))dt+\sigma(Y(t))dB(t)-d\eta_Y(t)+\frac{1}{\xi_t}\sigma(Y(t))\sigma(X(t))^{-1}(X(t)-Y(t))dt,\quad Y(0)=y .
\end{eqnarray*}
 $(X(t),Y(t))$ is a well defined continuous process for $t<T\wedge \zeta$, where $\zeta=\lim_{n}\xi_n$ for
$$\zeta_n:=\inf\{t\in [0,T):|Y(t)|\ge n\},$$
with convention $\inf\emptyset=T$. Let
$$R_{t}:=\exp\left[ -\int_0^{t\wedge \zeta}\frac{1}{\xi_s}<\sigma(X(s))^{-1}(X(s)-Y(s)),dB(s)>-\frac{1}{2}\int_0^{t\wedge\zeta}\frac{1}{\xi_s^2}|\sigma(X(s))^{-1}(X(t)-Y(t))|^2ds\right]$$
 $t\in [0,T)$.\\

$(R_{t})_{t\in[0,T]}$ is a uniformly integrable martingale (see \cite{wang}), and then the process
 $$d\tilde{B}(t)=dB(t)+\frac{1}{\xi_t}\sigma(X(t))^{-1}(X(t)-Y(t))dt$$
 is a Brownian motion under the new probability $\mathbb{Q}=R_T\,\mathbb{P} $. Consequently the processes $X$ and $Y$
 satisfy under $\mathbb{Q}$
 \begin{eqnarray*}
 dX(t)&=& b(X(t))dt+\sigma(X(t))d\tilde{B}(t)-d\eta_X(t)-\frac{X(t)-Y(t)}{\xi_t}dt,\quad X(0)=x\\
 dY(t)&=& b(Y(t))dt+\sigma(Y(t))d\tilde{B}(t)-d\eta_Y(t),\quad Y(0)=y .
 \end{eqnarray*}
 Taking into account that $X(T)=Y(T)$ \, $\mathbb{Q}-a.s.$,
 Young inequality gives rise for any $f\ge 1$
\begin{eqnarray*}
 P_T\log f(y)&=&\mathbb{E}_{\mathbb{Q}}\log f(Y_T)=\mathbb{E}\left[R_{T}
 \log {f(X_T)} \right]\\
 &\le& \mathbb{E}R_{T}\log R_{T}+\log\mathbb{E}f(X(T))\,.
\end{eqnarray*} 
Now, using the following estimation given by Lemma 2.1 in \cite{wang}
 $$\mathbb{E}R_{T}\log R_{T}\leq \frac{-\delta|x-y|^2}{\lambda^2\theta(2-\theta)(1-e^{2\delta T})} $$
 and taking $\theta=1$, we complete the proof of the first inequality.\\
 On other hand, let $\theta=\theta_T$, since $X(T)=Y(T)\,\, \mathbb{Q}-a.s$,
 we have
 \begin{eqnarray*}
 \left( P_Tf(y)\right)^p&=&\left( \mathbb{E}_{\mathbb{Q}}f(Y(T))\right)^p=\left( \mathbb{E}R_{T}
 f(X(T))\right)^p \\
 &\le& P_Tf^p(x)\left(\mathbb{E}R^{p/p-1}_{T} \right)^{p-1}. 
 \end{eqnarray*} 
By equality in (\ref{Theta_T}) we see that
 $\frac{p}{p-1}=1+\frac{\lambda^2\theta_T^2}{4c_p(c_p+\theta_T\lambda)} =1+r_T$.
Lemma 2.2 in \cite{wang}, yields
 \begin{eqnarray*}
 \left(\mathbb{E}R^{p/p-1}_{T} \right)^{p-1}&=&\left(\mathbb{E}R^{1+r_T}_{T} \right)^{p-1}\\
 &\le& \exp\left[ \frac{-(p-1)\theta_T\delta(2k+\theta_T\lambda)|x-y|^2}{4k^2(2-\theta_T)(k+\theta\lambda)(1-e^{2\delta T})}\right]\\
 &=&\exp\left[\frac{-\delta\sqrt{p}(\sqrt{p}-1)|x-y|^2}{2c_p[(\sqrt{p}-1)\lambda-c_p](1-e^{2\delta T})}\right].
 \end{eqnarray*}
This completes the proof of the second inequality.
\end{proof}
\begin{remark}
As it is shown in \cite{wang} for non-reflected diffusions case,
we can apply Theorem \ref{Harnack1} to get some Harnack inequalities and
contractivity properties for transition probabilities of reflected
diffusion semigroups.

\end{remark}
\section{ Stochastic differential equations involving local times}

  As was pointed out in the introduction, in this section we consider
  the following stochastic differential equation (SDEL):
\begin{equation}\label{RSDE with local time}
\left\{
\begin{aligned}
dX(t)&=b(X(t))dt+\sigma(X(t))d{B}(t)+\int_{\mathbb{R}}\nu(da)\,dL_t^a(X)\\
X(0)&=x\,,\\
\end{aligned}
\right.
\end{equation}
where $b, \sigma:\mathbb{R}\longrightarrow \mathbb{R}$ are measurable functions and
$\nu$ is a bounded signed measure on $\mathbb{R}$, such that $|\nu(a)|<1,\,
\forall a\in \mathbb{R}$. The process $L_t^a(X)$ stands for the symmetric local time of
the unknown process $X$ at a point $a$ and $B_t$ is a real Brownian motion defined on a
complete filtered probability space $(\Omega,\mathcal{F},\mathbb{P},(\mathcal{F}_t)_{t\geq0} )$,
where $(\mathcal{F}_t)_{t\geq0}$ is the natural filtration generated by
$B_t$.
SDEs of type $(\ref{RSDE with local time})$ have been studied previously
by many authors. We will refer essentially to the paper of
\cite{LeGall 1983} where the author gives necessary and sufficient conditions for pathwise
uniqueness, and together with results on the existence of weak solutions
he also proves results on strong uniqueness.  

Consider the function $f_\nu$ given by:
\begin{equation*}
f_\nu(x)=\exp\Big(-2\nu^c(]-\infty,x])\Big)\underset{y\le x}{\prod}\Bigg(\frac{1-\nu({y})}{1+\nu({y})}\Bigg)\,,\,\,\,\,\,\,\, x\in\R
\end{equation*}
where $\nu^c$ is the continuous part of $\nu$. The following lemma
appears in \cite{LeGall 1983}.

\begin{lemma}\label{le Gall lemma}
	Let $\nu$ be a bounded signed measure on $\mathbb{R}$. Then we have:
\begin{enumerate}
	\item $f_\nu$ is of bounded variation on $\mathbb{R}$,
	\item $f_\nu$ is right continuous,
	\item  there exist constants $ m,\ M >0 $ such that $ m\leq f_\nu \le M$,
	\item the function $f_\nu$ satisfies $f_\nu'(dx)+\Big(f_\nu(x)+f_\nu(x-)\Big)\nu(dx)=0$, with $f_\nu'(dx)$ denotes the bounded measure associated with $f_\nu$ and $f_\nu(x-)$ denotes the left-limit of $f_\nu $ at a point $x$.
\end{enumerate}
\end{lemma}
We consider the function $
F(x)=\int_0^x\, f_\nu(u)\,du $, $x\in\R$. It is easy to show that $F$ is one to one and that $F$ and 
$F^{-1}$ are Lipschitz functions.
An appeal to It\^o-Tanaka formula provides us the following lemma
(see \cite{LeGall 1983}) 
\begin{lemma}\label{transf}
	A process $X$ is a solution of equation $(\ref{RSDE with local time})$ if and only if \  $Y:=F(X)$ is a solution of:
	\begin{equation}\label{OSDE}
	dY(t)=\bar{b}(Y(t)) \, dt +\bar{\sigma}(Y(t))\, dB(t)\,,\,\, Y(0)=F(x)\,,
	\end{equation}
	where 
	$$\bar{b}(x)=(bf_\nu)\circ F^{-1}(x)\quad  and\quad \bar{\sigma}(x)=(\sigma f_\nu)\circ F^{-1}(x).$$
\end{lemma}

In the litterature, existence and uniqueness results of strong solutions for equations of type
(\ref{RSDE with local time}) and (\ref{OSDE}) were obtained under weaker conditions than we will assume here.
But, the proof of 
$ T_{2}(C)-$inequality requires some strong conditions on the coefficients (even in  ordinary SDEs framework) . In the rest of this section we assume that there is a unique strong solution of (\ref{OSDE}) .

To prove our results we need the following
stability property of $T_2(C)$ 
(see Lemma 2.1 in \cite{djellout}):
\begin{lemma}\label{stabil} Let $(E,d_E)$ and $(F,d_F)$ be two metric spaces and
$\psi:(E,d_E)\to(F,d_F)$ is a Lipschitz
application ($\alpha>0$), such that for an $\alpha>0$
	$$d_F(\psi(x),\psi(y))\le \alpha\, d_E(x,y)\quad \forall x,\, y\in\, E\,.$$
If $\mu \in T_p(C)$ on $(E,d_E)$,
then $\tilde{\mu}:=\mu\circ\psi^{-1}\, \in\, T_p(\alpha^2C)$ on $(F,d_F)$, for any $p\geq 1$.
\end{lemma}
 The next theorem is the main result of this section.
 We need to suppose that $ \bar{b}\,\,\text{and}\,\, \bar{\sigma}$ are globally Lipschtz functions, that is we make the following assumptions on the coefficients of equation $(\ref{RSDE with local time})$:
$${\bf H(6)}\quad \quad |\bar{b}(x)-\bar{b}(y)|
\vee
	|\bar{\sigma}(x)-\bar{\sigma}(y)|\le k\, |x-y|\,,\,\,
	\forall x,y\in \mathbb{R}\,.$$
We also suppose the following dissipativity assumption:\\
 $\qquad\quad {\bf  H(7)}$\quad \quad  there exists $\delta >0$ such that, for all $x,y\ \in \mathbb{R}$ we have
 $$\Big[(\bar{\sigma}(x)-\bar{\sigma}(y)\Big]^2+2(x-y)\big(\bar{b}(x)-
 \bar{b}(y)\big)\le -2\delta|x-y|^2.$$
\begin{theorem}\label{$T_2$ w.r.t $d_2$} 
	Suppose that $H(6)$ and $H(7)$ hold and
	$\vert\vert\bar{\sigma}\vert\vert_{\infty}<\infty $.
	Let $\mathbb{P}_X$  be the law of $X$, the solution of the
	stochastic differential
	equation $(\ref{RSDE with local time})$ with initial condition
	$X(0)=x\in{\mathbb R} $, then we have
	\begin{enumerate}
	 \item[1--] The probability measure $\mathbb{P}_X$ satisfies
	 $T_{2}(\frac{\|\bar{\sigma}\|_\infty^2}{m^2\delta^2})$ on the metric space $\mathcal{C}([0,T],\mathbb{R})$
	 equipped with the metric $d_2$.
	\item[2--] There exists some constant
	$ C=C(T, \|\bar{\sigma}\|_\infty , m, k)>0$ such
	that $\mathbb{P}_X\in\, T_2(C)$ on $\mathcal{C}([0,T],\mathbb{R})$
	with respect to $d_\infty$-metric.
	\end{enumerate}
\end{theorem}
\begin{proof}
By notations of Lemma \ref{transf}, let us
	consider the following ordinary stochastic differential
	equation:
	\begin{equation}\label{eqtransf}
	\left\{
	\begin{aligned}
	dY(t)&=\bar{b}(Y(t))\, dt +\bar{\sigma}(Y(t))dB(t),\\
	Y(0)&=F(x).\\
	\end{aligned}
	\right.
	\end{equation}
	Following the same arguments as those used in section 2, we can
	prove under assumptions of Theorem \ref{$T_2$ w.r.t $d_2$} that
$\mathbb{P}_Y\in T_2(\frac{\|\bar{\sigma}\|^2_\infty}{\delta^2})$ with respect to
the $L^2 $ norm (we can also see \cite{djellout}). We now consider the application $\Psi$ defined by:
$$\begin{array}{ccccc}
\Psi & : & \mathcal{C}([0,T],\mathbb{R}) & \to & \mathcal{C}([0,T],\mathbb{R}) \\\label{map psi}
& & \gamma & \mapsto & \psi(\gamma)=F^{-1}\circ\gamma\,.\\
\end{array}$$
It's clear that for each $\gamma_1,\,\gamma_2\,\in \mathcal{C}([0,T],\mathbb{R})$
$$d_2(\Psi(\gamma_1),\Psi(\gamma_2))\le \frac{1}{m}\,d_2(\gamma_1,\gamma_2),$$
thus, the map $\Psi$ is $\frac{1}{m}$-Lipschitzian, where $m$ is provided
by Lemma $(\ref{le Gall lemma})$. \\
In other hand, we have
$$\mathbb{P}_X=\mathbb{P}_Y\circ\Psi^{-1}\,,$$
and by stability property of $T_2(C)$ under Lipschitzian maps (Lemma
\ref{stabil}), we conclude that $\mathbb{P}_X\in T_2(\frac{\|\bar{\sigma}\|^2_\infty}{\delta^2 m^2})$
on $\Big(\mathcal{C}([0,T],\mathbb{R}),d_2\Big)$. Which ends the proof
of the first assertion.  The second point uses the same arguments.
\end{proof}

\begin{remark}

 If we take, $b=0$, $\sigma=1$ and $\nu=\beta\delta_0$, where
 $|\beta|<1$, we recognize the famous Skew Brownian motion for
 which $H(7)$ is no longer valid. and we cannot get the property
 $T_2(C)$ (or $T_{1}(C)$) for this process form the Theorem
 \ref{$T_2$ w.r.t $d_2$}. To our knowledge, this question has not
 yet been adressed in the literature. The only related result we
 found is in \cite{abakirova} where some Poincar\'e and log-Sobolev type
 inequalities have been highlighted.   
\end{remark}
 \begin{remark}
 	\begin{enumerate}
 		\item An appeal to Jensen inequality yields that $T_2(C)\Rightarrow T_1(C)$, then the property $T_1(C)$ holds for the probability measure $\mathbb{P}_X$ and we have for any Lipschitzian function $G:\mathcal{C}([0,T],\mathbb{R})\to \mathbb{R}$
 		\begin{equation}\label{concentration inequa by T1}
 		\mathbb{P}_X\big(G-\mathbb{E}_{\mathbb{P}_X}G\ >\ r\big)\ \le\ \exp\Big(-\frac{r^2}{2C\|G\|_{Lip}^2}\Big),\quad \forall\ r>0.
 		\end{equation}
 		
 		Let $V:\mathbb{R}\to \mathbb{R}$ be a Lipschitzian function, such that $\|V\|_{Lip}\le \alpha$. We define $F_V$ and $F_\infty$ on $\mathcal{C}([0,T],\mathbb{R})$ by 
 		$$F_V(\gamma)=\frac{1}{T}\int_{0}^{T}V(\gamma(t))\ dt,$$
 		$$F_\infty(\gamma)=\underset{t\in[0,T]}{\sup}|\gamma(t)-\gamma(0)|.$$
 		The function $F_V$ is $\alpha-$Lipschitzian with respect to $d_\infty$. As for $F_\infty$, it's $1-$Lipschitzian map with respect to $d_\infty$. Using $(\ref{concentration inequa by T1})$ we have the following Hoeffding-type inequalities for the solution $X$ of $(\ref{RSDE with local time})$ on the metric space of continuous functions, endowed with the metric $d_\infty$. For all $r>0$ we have
 		\begin{equation}\label{smalltime}
 		\mathbb{P}\big(\frac{1}{T}\int_{0}^{T}V(X(t))-\mathbb{E}V(X(t))\ dt\ >\ r\big)\ \le\ \exp\Big(-\frac{r^2}{2C\alpha^2}\Big),
 		\end{equation}
 		and using the functional $F_\infty$ we get
 		\begin{equation}\label{largetime}
 		\mathbb{P}\Big(\underset{t\in [0,T]}{\sup}|X(t)-x|-\mathbb{E}\Big[\underset{t\in[0,T]}{\sup}|X(t)-x|\Big]\ >\ r\Big)\ \le \ \exp\Big(-\frac{r^2}{2C}\Big).
 		\end{equation}
 		\item The estimates $(\ref{smalltime})$ and $(\ref{largetime})$ are well adapted to the study of small and large time asymptotics of the solution of equation $(\ref{RSDE with local time})$.  
 	\end{enumerate}
 	 \end{remark}
 	The solution $X$ of SDEL (\ref{RSDE with local time}) is a
 	strong Markov process. Let $(P_t)$ be the semigroup of
 	transition probability kernels of our diffusion. The next
 	proposition shows the existence of a unique invariant measure.
 \begin{proposition}
 	Under the assumption $H(7)$ there exist a unique invariant
 	measure $\mu$ for $(P_t)$, and we have the following
 	exponential convergence in sense of Wasserstein distance:
 	$$W_2(P_t(x,.),\mu)\le \frac{M}{m}e^{-\delta t}\Big(\int
 	|x-y|^2\mu(dy)\Big)^{\frac{1}{2}},
 	\,
 	\forall x\in \mathbb{R},\ t>0.$$
 \end{proposition}
 \begin{proof}
 Let $Q_t(x,.)$ denote the transition kernels associated to the
 Markov process solution of equation (\ref{OSDE}), by \cite{djellout}, $(Q_t)$
 admit a unique invariant measure $\tilde{\mu}$, thanks to the transformation
 $X_t=F^{-1}(Y_t)$, we have $$P_tf(x)=Q_t(f\circ F^{-1})(F(x)).$$
 Then it's easy to check that $\mu:=\tilde{\mu}\circ F$ is the unique
 measure invariant for $(P_t)$.    
 Again by \cite{djellout}, we get  
 \begin{eqnarray*}
 W_2(P_t(x,.),\mu)&=& W_2(Q_t(F(x),F(.)),\tilde{\mu}\circ F)\\
                  &\le& \frac{1}{m} W_2(Q_t(F(x),.),\tilde{\mu})\\
                  &\le& \frac{1}{m}e^{-\delta t}\Big(\int |F(x)-y|^2\tilde{\mu}(dy)\Big)^{\frac{1}{2}}\\
                  &\le& \frac{1}{m}e^{-\delta t}\Big(\int |F(x)-F(F^{-1}(y))|^2\tilde{\mu}(dy)\Big)^{\frac{1}{2}}\\
                  &\le& \frac{M}{m}e^{-\delta t}\Big(\int|x-y|^2 \mu(dy)\Big)^{\frac{1}{2}}.                  
 \end{eqnarray*}
 Which completes the proof.
 \end{proof}
 The next theorem show a Harnack inequality for the semigroup of $X$, which is a consequence of the Theorem $1.1$ proven in \cite{wang}  for ordinary stochastic differential equation under the additional assumptions:
 \begin{equation*}
 	H(8)\qquad \text{There exist}\,\, \lambda\,,\, \beta>0\,,\,\,
 	s.t.\,\,\,\,\bar{\sigma}(x)^2\ge
 	\lambda,\quad\text{and}\quad|\left(
 	\bar{\sigma}(x)-\bar{\sigma}(y)\right)(x-y)|\le \gamma |x-y|,
 	\quad x,y\in \R,
 \end{equation*}

 \begin{theorem}\label{harnack2}
 	If $H(7)$ and $H(8)$ hold, then for $p>(1+\frac{\gamma}
 	{\lambda})^2$ and $\beta_{p}=\max\{\gamma,\frac{\lambda}{2}(
 	\sqrt{p}-1)\}$, the Harnack inequality
 	\begin{equation*}
 		\left(  P_Tf(y)\right)^p\le \left(P_Tf^p(x)\right)
 		\exp\left[\dfrac{-\delta M \sqrt{p}(\sqrt{p}-1)
 		|x-y|^2}{2\gamma_p[(\sqrt{p}-1)\lambda-\gamma_p]
 		(1-e^{\delta T})} \right]   
 	\end{equation*}
 	holds for all $T>0$, $x,y\in \mathbb{R}$ and f positive bounded measurable function on $\R$.
 \end{theorem}
 \begin{proof}
 	According to Theorem $1.1$ in \cite{wang}, the semigroup $Q$
 	associated to the transformed equation $(\ref{OSDE})$ satisfies
 	the Harnack inequality, and by the relation $P_tf(x)=Q_t(f\circ F^{-1})(F(x))$
 	we deduce the desired inequality. 
 \end{proof}
 \begin{example} For $\delta>0$, $ 0< \beta <1 $, we consider Equation (\ref{RSDE with local time}) driven by the coefficients:
 $$\sigma(x)=I_{x<0}+\frac{1+\beta}{1-\beta}I_{x\geq 0} \,\,\,;\,\, b(x)= -\delta x\, I_{x<0} -\delta x\,\frac{1+\beta}
  {1-\beta}I_{x\geq 0} $$
  and the measure $\nu=\beta\,\delta_{0} $.\\
  We get easily ${\bar\sigma}(x)=1 $,   ${\bar b}(x)=-\delta x $, $m=\frac{ 1-\beta}{1+\beta} $ and $M=1 $. The corresponding process
  $Y$ solution of Equation (\ref{OSDE}) is an Ornstein-Uhlenbeck process and that $\bar\sigma $, $\bar b $ satisfy the conditions of Theorem \ref{$T_2$ w.r.t $d_2$} and Theorem \ref{harnack2}.

\end{example}


\end{document}
\endinput

\end{document}